\pdfoutput=1
\documentclass[ a4paper, onecolumn, superscriptaddress, 10pt, nopdfoutputerror
              , accepted=2022-02-14, issue=8, volume=5, shorttitle=papers/compositionality-5-8
              ]{compositionalityarticle}

\usepackage[utf8]{inputenc}
\usepackage[english]{babel}
\usepackage[T1]{fontenc}
\usepackage{amsmath}
\usepackage{amsthm}
\usepackage{amssymb}
\usepackage{stmaryrd}
\usepackage{epigraph}

\usepackage[
            maxbibnames=9,
            backref=true,
            sorting=nyt]{biblatex}
\usepackage{hyperref}

\usepackage[all,cmtip]{xy}

\theoremstyle{definition}
\newtheorem{definition}{Definition}
\newtheorem{example}[definition]{Example}
\newtheorem{theorem}[definition]{Theorem}
\newtheorem{corollary}[definition]{Corollary}
\newtheorem{remark}[definition]{Remark}

\DeclareMathOperator{\id}{id}
\DeclareMathOperator{\pr}{pr}
\DeclareMathOperator{\ev}{ev}

\newcommand{\idx}{\mathsf{idx}}

\newcommand{\fix}{\mathsf{fix}}
\newcommand{\app}{\mathsf{app}}
\newcommand{\lam}{\mathsf{lam}}
\DeclareMathOperator{\disc}{disc}
\newcommand\CC{\mathsf{C}}

\newcommand\NN{\mathbb{N}}
\newcommand\Set{\mathsf{Set}}
\newcommand\Top{\mathsf{Top}}

\newcommand\Bool{\texttt{Bool}}
\newcommand\onto{\twoheadrightarrow}

\begin{document}

\title{Substructural fixed-point theorems and the diagonal argument: theme and variations}
\author{David Michael Roberts}
\email{david.roberts@adelaide.edu.au}
\homepage{https://ncatlab.org/nlab/show/David+Michael+Roberts}
\orcid{0000-0002-3478-0522}
\thanks{This work was supported by the Australian Research Council’s Discovery Projects funding scheme (grant number DP180100383), funded by the Australian Government.}
\affiliation{School of Computer and Mathematical Sciences, The University of Adelaide, Adelaide, Australia}
\maketitle

\begin{abstract}
This article re-examines Lawvere's abstract, category-theoretic proof of the fixed-point theorem whose contrapositive is a `universal' diagonal argument. The main result is that the necessary axioms for both the fixed-point theorem and the diagonal argument can be stripped back further, to a semantic analogue of a weak substructural logic lacking weakening or exchange.
\end{abstract}

\setlength{\epigraphwidth}{0.8\textwidth}
\epigraph{%
\emph{Dieser Beweis erscheint nicht nur wegen seiner grossen Einfachheit, sondern namentlich auch aus dem Grunde bemerkenswert, weil das darin befolgte Princip sich ohne weiteres  \ldots ausdehnen l\"asst, \ldots}\\%
This proof appears remarkable not only because of its great simplicity, but also for the reason that its underlying principle can readily be extended,\ldots}{Georg Cantor}

Little could Cantor have foreseen exactly \emph{how} far the principle of his simple proof \cite{Cantor_1892} would be extended.
The technique of diagonalisation has appeared in many places, and this was captured by an abstract category-theoretic proof of a diagonal argument by Lawvere \cite{Lawvere_69}, applying to not just bare sets and functions, but more structured objects.
A key ingredient of Cantor's original proof\footnote{It is perhaps telling that Cantor didn't originally run the diagonal argument on a list of real numbers expressed as decimal expansions, as usually presented. He was working with sequences of the symbols $w$ and $m$ instead.}---which showed the set of functions $\NN\to \{m,w\}$ is uncountable---is the existence of the function $\{m,w\}\to \{w,m\}$ that sends each element to the other.
In Lawvere's abstract setting, it is possible to have structured objects $C$ with no such `free endomorphisms': given $\sigma\colon C\to C$, there always exists some element unmoved by $\sigma$. 
Lawvere's fixed-point theorem constructs this element, and the diagonal argument is then the contrapositive of the fixed-point theorem.

Yanofsky \cite{Yanofsky_03} gave an extensive survey of Lawvere's pair of results, showing how a great many examples from the literature follow as special cases, for example Cantor's original theorem on cardinalities, Russell's paradox, G\"odel's first incompleteness theorem,\footnote{Gromov \cite[\S2.1]{Gromov_ergosys} gives an independent treatment of G\"odel's theorem making it obviously a special case of Lawvere's result, with the comment: ``The childish simplicity of the proof of G\"odel's theorem does not undermine its significance''.} Tarski's theorem on the non-definability of a truth predicate, the Halting problem,\ldots. 
There are hints by Yanofsky, however, that Lawvere did not state the results in the most generality, even in subsequent iterations \cite[Session 29.2]{Lawvere-Schanuel} that weakened an assumption.
Taking such hints seriously, and looking for the actual minimum assumptions needed for Lawvere's proof to apply, we find that even the basic assumption of cartesian products in all existing treatments is not needed.
As a result, certain diagonal-type arguments (and fixed-point results) can be constructed that escape capture by Lawvere's formulation.

\begin{example}
Consider the following setup. We are dealing with countably infinite sets, and only considering functions $f\colon A\to B$ that are finite-to-one: for each $b\in B$, there are only finitely many $a\in A$ with $f(a)=b$.
The cartesian product of countable sets is of course again countable\footnote{Assuming for this example the axiom of choice.}, and we can define for every set $A$ a diagonal map $A\to A\times A$ sending $a\mapsto (a,a)$. Importantly, the projection maps $A\times A\to A$, $(a_1,a_2) \mapsto a_i$ are \emph{not} finite-to-one, which implies that the  product here is \emph{not} the categorical product assumed to exist by the statement of Lawvere's diagonal argument. This setup describes a category with a notion of product, specified in more detail below.

Yet \emph{a} diagonal argument still works in this setting. Consider for simplicity a finite-to-one function $F\colon A\times A \to \NN$. And then the finite-to-one function $A\to \NN$, $a\mapsto F(a,a)+1$, is not equal to $F(a',-)\colon A\to \NN$ for any $a'\in A$.
As a result, no such finite-to-one function $F$ can capture all possible finite-to-one functions $A\to \NN$.
\end{example}

So what should we do? Try to embed this example in a setting where Lawvere's result applies verbatim, and then argue back down to the desired statement? 
In the author's view, this approach may be limiting, especially when working internal to a given ambient category with a weak internal logic; having a direct argument can be preferable to relying on an embedding theorem.
One can examine the proof and extract the relevant definitions that make the result work, as is done in this article, albeit in several directions at once, and for Lawvere's fixed-point theorem as well.

\paragraph{Outline} We begin the rest of the paper by giving a more in-depth introduction in \S1 to the cloud of ideas from which this paper coalesced, and touch on the relation to substructural logic.
This is followed by two sections that together are an exploration of the categorical framework the results of the paper will rely on: the first of these, \S2 gives the required definitions, followed by \S3 with a collection of nontrivial examples that are genuinely weaker than the cartesian framework used by Lawvere.
The next two sections (\S\S4--5) cover the diagonal argument and the fixed point theorem in this initial minimalist setting.
Then \S6 re-examines Lawvere's result in light of Yanofsky's choice to only treat cartesian categories of structured sets, and find that Lawvere's version is ultimately no more general than Yanofsky's.
In \S7 the results then turn to discussing the fixed-point theorem \emph{internal} to a regular category; the results preceeding this have all been from an \emph{external} point of view. In this section the fixed-point theorem is written out in the relevant Kripke-style semantics in the \emph{regular logic} fragment. This is not as weak as the other types of logics earlier in the paper, but the proof indicates that not all the structural rules are strictly necessary.
The penultimate section, \S8, considers a substructural setting analogous to lambda calculus, where one can build fixed-point combinators on function types, with the additional richness of allowing a necessity-type modal operator that seems to capture a fragment of allowing access to an oracle. This last interpretation is inspired by the `flat' modality from \emph{crisp type theory} and its use in the proof of Brouwer's fixed point theorem using `real cohesion'. I do not claim to be able to unify these two classes of fixed point theorem here, however.
Lastly, a short \S9 returns full circle to give an analogue of a fixed-point combinator for a reflexive object modelling $\lambda$-calculus, albeit in the much weaker substructural setting set up at the beginning.

\section{More technical introduction} 
\label{sec:detailed_introduction}

Lawvere's fixed-point theorem \cite{Lawvere_69} 
is a result about cartesian closed categories whose contrapositive is an abstract, categorical version of Cantor's diagonal argument.
It says that if $A \to Y^A$ is surjective on global points---every $1\to Y^A$ is a composite $1\to A\to Y^A$---then for every endomorphism $\sigma\colon Y\to Y$ there is a fixed (global) point of $Y$ not moved by $\sigma$.
However, Lawvere goes further and observes that the proof of his theorem does not really use the \emph{closed} structure affording the `internal hom' objects $Y^A$, the analogues of function sets.
The internal hom only appears in the theorem \emph{statement}; the first step in the proof is to uncurry a given map $A\to Y^A$ to the corresponding $A\times A\to Y$, from which the fixed point is constructed.
It should be noted that the diagonal copy of $A$ inside $A\times A$ plays a crucial r\^ole in the construction.
Thus with appropriate unwinding of the definitions involved, the theorem makes sense in an arbitrary category with finite products: a \emph{cartesian} category \cite[Session 29.2]{Lawvere-Schanuel}. 
This point of view is taken by Yanofsky \cite{Yanofsky_03}, where other small variations are also made to the setup, but still in the setting of a cartesian category.

On the face of it, this seems like the ultimate version of the diagonal argument: one needs the terminal object $1$ (a $0$-ary product) to talk about global points, and the binary products to talk about the function $A\times A \to Y$ to which diagonalisation is applied.
However, closer inspection of the proof makes it clear that it never uses the projection maps $\pr_1,\pr_2\colon A\times A\to A$ that characterise cartesian categories among the more general \emph{monoidal} categories.
Further, the universal property of the product is never used aside from ensuring the existence of the natural diagonal map $A\to A\times A$.

Thus one might ask if the proof makes sense in a more general monoidal category, with additional, explicitly-specified data playing the r\^ole of the diagonal maps. This is mentioned by Abramsky and Zvesper \cite[footnote 2]{Abramsky_Zvesper}, for example.
It does indeed, and Lawvere's proof in the cartesian case applies, practically verbatim, to more general monoidal categories with diagonals.
A second, even closer, inspection shows that the monoidal coherence structure, the data consisting of the isomorphisms $(A\otimes B)\otimes C \simeq A\otimes (B\otimes C)$ and so on, is never used.
Similarly, the first instinct is to replace the terminal object $1$ with the more general monoidal unit $I$, since the terminal object is the monoidal unit in the cartesian case. 
But the proof uses no monoidal properties of $I$ whatsoever, and so one can just fix an arbitrary object $t$ to play the domain of `generalised elements' $t\to A$.\footnote{For instance one can consider the category of partial functions $\NN \rightharpoonup \NN$, where there is no terminal object, and take $t=\NN$.} 
In this note we will see what seems like the minimal structure that supports both of Lawvere's results, together with variations on the theme of working with weaker assumptions.

It is interesting to compare the categorical framework here to the combinatory logic of Sch\"onfinkel and Curry.
One can consider combinators as operators on the set of $\lambda$-expressions\footnote{Technically, they are themselves \emph{closed} $\lambda$-expressions, and they operate via application.}, and a \emph{fixed-point combinator} $f$ satisfies
\[
    x(f x) = f x
\]
for all $\lambda$-expressions $x$---the expression $f x$ is a fixed-point of $x$.
The most famous of these is probably the \emph{$Y$-combinator}.
There are various small sets of combinators that are known to be able to generate all possible combinators/terms, for instance $S$, $K$ and $I$, generating the so-called $\mathit{SKI}$-calculus. 
But there are other options, and various constructions of fixed-point combinators from them.
More interesting for our purposes is the possibility of looking at combinators that do not generate the full $\mathit{SKI}$-calculus, but which are still enough to give a fixed-point combinator. 

Smullyan \cite{Smullyan_85} raised the question whether it was possible to do so starting only from the combinators $B$ and $W$, given by
\begin{align*}
((Bx)y)z &= x(yz)\\
(Wx)y &= (xy)y,
\end{align*}
where $x$ and $y$ are $\lambda$-terms, amounting to having composition and diagonals respectively. 
This problem was solved by Statman (reported by McCune and Wos \cite{McCune-Wos_87}), giving the explicit fixed-point combinator $f := (B(WW))((BW)((BB)B))$.

\begin{remark}
There is a hierarchy of substructural logics (eg \cite{Ono_90}), each of which corresponds to a particular choice of basis combinators. 
Relevance logic according to \emph{op.\ cit.} corresponds to $\mathit{BCWI}$-logic, linear logic corresponds to $\mathit{BCI}$-logic, and ``ordered logic'' to $\mathit{BI}$-logic.
Statman's combinator shows that one can construct a fixed-point combinator in $\mathit{BW}$-logic, which corresponds to a system called ``full Lambek with contraction'' $\mathit{FL}_c$, though one might also call it ``ordered logic with contraction''.
It is somewhat remarkable that $\mathit{FL}_c$ is undecidable \cite{Chvalovsky_Horcik_16},
 and is very nearly the minimal substructural logic for which this is true.\footnote{The result of Chvalovsk\'{y} and Hor\v{c}\'{i}k is about the `positive fragment' of $\mathit{FL}_c$.} 
The proof here of the fixed-point theorem from a semantic/categorical point of view is then intriguing, given that it suffices to prove incompleteness results \cite{Yanofsky_03}.
\end{remark}

\paragraph{Acknowledgements}  After I announced preliminary results from this note online, David Jaz Myers shared some talk slides outlining a graphical proof of the fixed-point theorem that continues to work in a monoidal category with sufficient diagonals. 
Thanks also to Greg Restall for answering questions on Twitter about relevance logic, and Rongmin Lu for extensive discussions and feedback.

\section{Products in a general setting}

This section serves to give the main technical definitions needed for the paper; the next section will give some examples.
As noted above, the projection maps afforded by finite products are not used in the proof of the diagonal argument or the fixed-point theorem, nor associativity up to isomorphism.
The only structure that appears to be needed is the following:

\begin{definition}{(Davydov \cite{Davydov_07})}
A \emph{pointed magmoidal category}
is a category $\CC$ equipped with a functor $\# \colon \CC\times \CC\to \CC$ and a chosen object $t$.
A magmoidal category is equipped with \emph{diagonals} when it is equipped with a natural transformation $\delta\colon \id_C \Rightarrow \#\circ \Delta_\CC$. 
\end{definition}

Another way to phrase the structure of diagonals on a magmoidal category is that every object $X$ is equipped with the structure of an \emph{internal comagma} (an arbitrary binary co-operation morphism $X\to X\# X$), and such that every morphism $X\to Y$ of the category is a map of comagmas.
The most general statement to be considered in this note will not need the full structure of diagonals on the magmoidal category, which can be considered as a global structure on $\CC$.
Instead, one can ask for local structure, namely that $t$ is equipped with a comagma structure, and the results concern a comagma object $A$ such that every morphism $t\to A$ is a comagma map.
This is a reasonable generalisation as there are magmoidal categories of interest, for instance monoidal categories with (non-natural) diagonals as described by Selinger \cite{Selinger_99}, where requiring natural diagonals for \emph{all} objects (known as ``uniform cloning'' in the setting of categorical quantum mechanics) leads to strong constraints on the monoidal structure (eg \cite{Abramsky_10}, in the closed setting).

For the purposes of section~\ref{sec:uniform} below, a little extra structure is needed.

\begin{definition}
Fix a magmoidal category $(\CC,\#)$ and a pair of objects $X,Y$. An \emph{internal hom object} $Y^X$ is a representing object for the functor $\CC(-\#X,Y)\colon \CC\to \Set$. Concretely, this means there is an object $Y^X$, an \emph{evaluation map} $\ev\colon Y^X\#X\to Y$ such that the assignment
\[
    (W\xrightarrow{f} Y^X) \mapsto (W\#X \xrightarrow{f\#\id} Y^X\#X \xrightarrow{\ev} Y)
\]
defines a natural isomorphism.
If we take $X=Y$, then $X^X$ will be called the \emph{internal endomorphism object}.
\end{definition}

One way to get such an internal hom object is to demand the magmoidal structure is (right) closed: for every object $X$ the functor $(-)\#X\colon \CC\to \CC$ has a right adjoint $(-)^X\colon \CC\to\CC$. 
This notion of closed magmoidal category was discussed with a view towards functional programming by Milewski \cite{Milewski_blog} under the name \emph{magmatic category}, and again by Wijnholds \cite{Wijnholds_17}, in the context of getting a graphical calculus for \emph{biclosed} magmoidal categories, in the style of Baez and Stay \cite{Baez-Stay_Rosetta}.

\section{Examples of magmoidal categories}\label{sec:examples}

Aside from cartesian monoidal categories, those where the product is the categorical product, there are a number of nontrivial examples, both of a general nature, and arising from the literature.
The rough idea is that one should be looking at categories that are able to give a semantic interpretation for some form of substructural type theory (eg \cite{Walker_05}), where one is not insisting on the Exchange or Weakening rules.

Such categories can be seen as generalising the algebraic semantics of relevance logic \cite{Restall_00}, for instance Church monoids \cite{Meyer_72} or Ackermann groupoids\footnote{That is, magmas, not categories where all morphisms are invertible.} with the additional rule $a \leq a\cdot a$ \cite{Meyer-Routley_72}.
Here $a\cdot b$ interprets intensional conjunction (or \emph{fusion}) in relevance logic, so one could see the product in a magmoidal category as giving a refinement, \`a la Curry--Howard propositions-as-types, to richer interpretations of relevant types.
Church monoids and Ackermann groupoids with diagonal give rise to partially ordered sets with an additional magma operation, hence magmoidal categories with diagonals, with at most one morphism between any two objects. 
The magmoidal categories arising from Church monoids are additionally closed, coming from residuation.

However, the partial orders that form algebraic semantics as studied in the relevance logic literature are a little too degenerate for our purposes.
The relevance monoidal categories of Do\v{s}en and Petri\'{c} \cite{Dosen-Petric_07} give examples of magmoidal categories with diagonal, though they are much more restrictive (cf also \cite{Szabo_78,Vasyukov_11} and \cite[\S 3.2.2]{Mere_phd}).
There is also ongoing work linking Dialectica-type categories to relevance logic that should form an example \cite{dePaiva_18,dePaiva_20}, though currently details are still forthcoming.

We mention in passing that arbitrary monoidal categories---and even skew monoidal categories \cite{Szlachanyi}---with diagonals are of course examples.

Now to a list of more concrete examples and constructions, to help anchor the reader's intuition.

\begin{example}
Consider a category with finite products, for instance the category of sets, or even just the category of countable sets.\footnote{In an ambient set theory without Choice, one should instead take \emph{counted} sets---those with a bijection with a subset of $\mathbb{N}$---otherwise finite products could fail to exist.}
Then the subcategory consisting of all the objects but just the monomorphisms---e.g.\ the category of (countable) sets and injections---with the cartesian product is magmoidal with diagonals.
The only projections that are guaranteed to exist are those mapping out of products with the terminal object.
\end{example}

Given the hovering presence of relevance logic, sets and injections seem like they capture some intuition about implications not being allowed to delete information.
However, this vague analogy will not be developed futher here.
If a given \emph{monoidal} category with diagonals always has the diagonal maps being monomorphisms, then this construction can be generalised to that case.

\begin{example}
Given a regular cardinal $\kappa$, a function $f\colon X\to Y$ between sets with $|f^{-1}(y)|< \kappa$ is called \emph{$\kappa$-small}. 
The category of sets and $\kappa$-small functions with the cartesian product of sets is magmoidal with diagonals, but the product is not the categorical product, as the projections are not generally $\kappa$-small.
For instance, $\aleph_0$-small functions are those where the preimage of any element in the codomain is finite.
One can think of such functions as only forgetting a `small' amount of information in their domain.
\end{example}

\begin{example}
The category of pointed sets and functions, with \emph{smash product}, is monoidal with diagonals.
Recall that the smash product of pointed sets $(X,x_0) \wedge (Y,y_0)$ is defined to be the quotient of $X\times Y$ by the equivalence relation where we declare all elements of the form $(x_0,y)$ and $(x,y_0)$ to be equivalent, for arbitrary $x$ and $y$.
The diagonal map for $(X,x_0)$ is given by the composite $(X,x_0)\to (X\times X,(x_0,x_0)) \to (X\wedge X,[x_0,x_0])$.

Recall that the category of pointed sets and pointed (total) functions is equivalent to the category $Set_{\mathrm{partial}}$ of sets and partially-defined functions, whereby the functor $Set_{\mathrm{partial}} \to Set_*$ takes a set $X$ and adds a new element, $X\sqcup \bot_X$, to be the base point, and for a partial function $f\colon X\to Y$ this functor sends all elements outside of the domain of $f$ to $\bot_Y$ (including $\bot_X$).
\end{example}

For a reasonably sophisticated example, consider the category of complete lattices (that is, posets with suprema of all subsets) with strict morphisms:  those preserving the bottom element and suprema of all directed subsets \cite{Jacobs_93}.
Another is the category of `worlds' in the sense of O'Hearn--Power--Takeyama--Tennent \cite{OPTT_95}, which is monoidal with diagonals \emph{and} projections---but not cartesian, as diagonal followed by either projection is not the identity morphism.

\begin{example}
Consider a symmetric monoidal category $(\CC,\otimes,I)$ (for instance that of vector spaces), and the category $\mathbf{CoSemigrp}_{\mathrm{cocom}}(\CC)$ of cocommutative cosemigroup\footnote{Recall that this means being equipped with a coassociative comultiplication, but no counit object in general.} objects in $\CC$.
This is monoidal, and the comultiplication map $\Delta\colon A\to A\otimes A$ is a map of cosemigroups.
Further, every map of cosemigroups $A\to B$ turns out to be compatible with comultiplication\footnote{The proof is the same as in the case of commutative comonoids in a symmetric monoidal category.}, and thus $\Delta$ is natural as needed. 
\end{example}

The category of cocommutative \emph{comonoids}\footnote{Hence with a two-sided counit.} in a symmetric monoidal category is cartesian monoidal, but removing the counital axiom means we can get a more general monoidal category, while retaining the diagonals.

The last general construction is one that starts from a monoidal category with diagonals, for instance a category with finite products.

\begin{example}
Suppose $T\colon \CC\to \CC$ is a pointed endofunctor equipped with a natural transformation $\iota\colon \id_{\CC} \Rightarrow T$, on a monoidal category $(\CC,\otimes)$ with diagonals (for instance a category with finite products). 
We can define a new magmoidal product ${}_T\#$ by $A{}_T\# B := T(A)\otimes B$.
This still has diagonals, since we can use $\iota_A$ to get $A\to A\otimes A\to T(A)\otimes A$.

Similarly one has a new magmoidal product $\#_T$, given by $A\#_T B := A\otimes T(B)$, which also has diagonals.
\end{example}

As a concrete example, one can consider the slice category $\Set/X$, together with the functor sending an object $(A\to X)$ to $(\pr_2\colon A\times X \to X)$.
The product on $\Set/X$ is the categorical product, given here by $(A\to X)\times (B\to X) :=(A\times_X B\to X)$
Then the new, magmoidal product is, on objects, $(A\to X){}_T\# (B\to X) = (A\times B\xrightarrow{\pr_2} B \to X)$.
There are diagonals, given by the usual diagonal $A\to A\times A$, taken as morphisms over $X$.
This example is even magmoidal closed, where the corresponding internal hom $(C\xrightarrow{g} X)^{(B\xrightarrow{f} X)}$ is given by
\[
     \pr_2\colon \hom_X(B,C)\times X\to X
\]
where $\hom_X(B,C)$ is the set of functions $B\to C$ over $X$.\footnote{Compare to the usual cartesian closed structure, in which the internal hom is the \emph{dependent product}, given by the obvious projection $\bigsqcup_{x\in X} g^{-1}(x)^{f^{-1}(x)} \to X$.}

Another variation is the full subcategory with objects the \emph{surjective} functions $A\onto X$.
Yet another is to take this variation and restrict to the injective functions over $X$.

A different example is given by the category of pointed sets with smash product, where the endofunctor $T$ is taken to be (on objects) $(X,x_0) \mapsto (X\sqcup \{\bot_X\},\bot_X)$.
Then the new magmoidal product is
\[
    (X,x_0)\#_T (Y,y_0):= (X\times Y)/((x_0,y)\sim (x_0,y'))
\]
This is magmoidal closed, where one takes the internal hom $(Y,y_0)^{(Z,z_0)}$ to consist of \emph{all} functions, and be pointed by the function constant at $z_0$.

I hope it is clear from this bevy of examples there are plenty of nontrivial instances to play with.

\section{The diagonal argument in the magmoidal setting}\label{sec:magmoidal_diagonal}

For the following we shall fix a pointed magmoidal category with diagonals $(\CC,\#,\delta,t)$.
In the cartesian setting, a map $A\times B \to C$ can be viewed as an $A$-parameterised family of maps $B\to C$; in the cartesian closed case this is of course equivalent to a map $A \to C^B$.
In our setting of a magmoidal category, we still want to think of a morphism $F\colon A\# B \to C$ as being an $A$-parametrised family of maps $B\to C$.

How does one see this?
Given any $a\colon t\to A$, hence supposedly picking out one such map, then for any $b\colon t\to B$ we get 
\[
F(a,b)\colon t\xrightarrow{\delta} t\# t \xrightarrow{a\# b} A \# B \xrightarrow{F} C\,.
\] 
So, if we think of $a$ as being an element of $A$, then we have an assignment of elements $b\mapsto F(a,b)$.

We can ask whether, for a given $A$, one can find \emph{every} map $B\to C$ inside some given $A$-parametrised family $A\# B \to C$, up to equality defined by considering $t$-points; this is a form of \emph{observational equivalence} relative to the specified object $t$.

\begin{definition}
In any magmoidal category with diagonals $(\CC,\#,\delta,t)$, a map $F\colon A\#B\to C$ is an \emph{incomplete parametrisation} of maps $B\to C$ if there exists an $f\colon B\to C$ such that for all $a\colon t\to A$, there is some 
$b \colon t\to B$ with $f\circ b \not= F\circ(a\# b)\circ \delta\colon t\to C$.
\end{definition}

\noindent

Finally, let us say that an endomorphism $\sigma\colon C\to C$ in a category is $t$-\emph{free} if for all $c\colon t\to C$, $\sigma\circ c \not=c$.
Notice that if there is a $t$-free endomorphism, then $t$ is not an inital object. 

\begin{theorem}\label{thm:diag_thm}
  For $(C,\#,\delta,t)$ a pointed magmoidal category with diagonals, and $\sigma \colon C\to C$ a $t$-free endomorphism, every $F\colon A\# A \to C$ is an incomplete parametrisation of maps $A\to C$.
\end{theorem}

\begin{proof}
Define $f = \sigma\circ F\circ \delta_A$. 
Then for all $a\colon t\to A$, we can take $b=a$, in the definition of incomplete parametrisation.
This gives the following (non-commutative!) diagram.
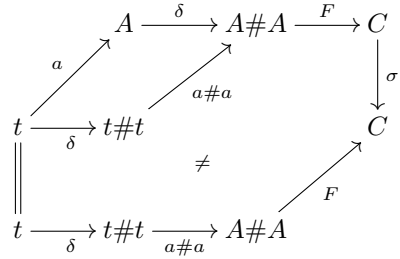
\[
\xymatrix{
&A \ar[r]^-\delta  & A\# A  \ar[r]^-F  & C\ar[d]^\sigma \\
t\ar[r]_\delta \ar@{=}[d] \ar[ur]^a  &t\#t \ar[ur]_{a\# a} \ar@{}[dr]^\neq& & C \\
t\ar[r]_\delta & t\#t \ar[r]_-{a\# a} & A\#A \ar[ur]_F
}\qedhere
\]
\end{proof}

This recovers the diagonal argument given by Lawvere, if we take $\CC$ to be a cartesian monoidal category, and $t=1$, the terminal object.

\begin{remark}
The argument of course is extremely `local', in that the only instance of naturality for $\delta$ needed is that all morphisms $t\to A$ are comagma morphisms, thinking of the diagonal maps as the structure of an internal comagma.
It is conceivable that $A$ is a special kind of object in a category where diagonals don't exist for all objects.
\end{remark}

There is a further variation, introduced by Yanofsky \cite{Yanofsky_03}, where one has an auxiliary map $p\colon A\to B$ with a section $s\colon B\to A$.

\begin{theorem}
In the situation of Theorem~\ref{thm:diag_thm}, given a morphism $p\colon B\to A$ with a section $s\colon A\to B$ (so that $p\circ s=\id_A$), every $F\colon A\# B \to C$ is an incomplete parametrisation of maps $B\to C$.
\end{theorem}

\begin{proof}
In the definition of incomplete parametrisation, we take $f\colon B\to C$ to be the composite $F\circ (\id \# p)\circ \delta_A \circ s$ and for a given $a\colon t\to A$, we take $b=s\circ a$:
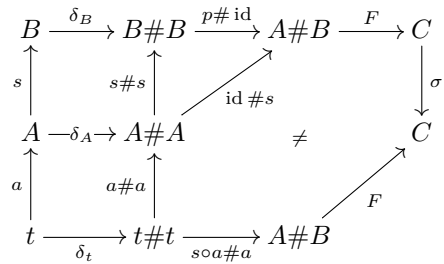
\[
    \xymatrix{
        B \ar[r]^-{\delta_B}  & B\# B \ar[r]^{p\#\id} & A\# B\ar[r]^-F  & C\ar[d]^\sigma \\
        A\ar[r]|-{\delta_A} \ar[u]^{s}  &A\#A \ar[u]^{s\# s} \ar[ur]_{\id\# s}="s" & & C \\
        t\ar[u]^a \ar[r]_-{\delta_t} & t\#t \ar[u]^{a\#  a} \ar[r]_-{s\circ a\# a} & A\#B \ar[ur]_F^{\ }="t"
        \ar@{}"s";"t"|{\neq}
    }
\]
where we have used naturality of $\delta$ twice, with respect to $a$ and $s$.
\end{proof}

\section{The fixed-point theorem in the magmoidal setting}\label{magmoidal_FPT}

We start with a version of the fixed-point theorem in the setting of the previous section, but will go on to prove some more variations below.

Yanofsky points out \cite[Remark 5]{Yanofsky_03} that one does not \emph{need} the global quantifier $\forall f\colon A\to B$ as Lawvere has in his definition of `weakly point surjective'---the proof requires quantifying over only those morphisms that factor as $F\circ \delta\colon A\to B$ followed by some $B\to B$.
Here we go one step further and remove this quantifier altogether, since the construction of the fixed point doesn't require universal quantification over any such collection of functions.

\begin{theorem}
Let $(\CC,\#,\delta,t)$ be a pointed magmoidal category with diagonals and $F\colon A\#A\to C$ and $\sigma\colon C\to C$ be maps such that
\begin{align*}
&\exists a_0\colon t\to A\\
&\forall a\colon t\to A\\
&\sigma\circ F\circ \delta_A\circ a = F\circ (a_0\#a)\circ \delta_t.
\end{align*}
Then the map $c:= F \circ \delta\circ a_0\colon t\to C$ satisfies $\sigma\circ c = c$.
\end{theorem}

One can think of this as saying: ``if there is some $a_0\in A$ such that $F(a_0,-)=\sigma\circ F\circ\delta_A$, then $\sigma$ has a fixed point'', and further, the fixed point is an explicit function of $a_0$.

\begin{proof}
Consider the following commutative diagram:
\[
  \xymatrix{
    t \ar@{=}[d] \ar[r]^{a_0} &A  \ar[r]^-{\delta_A}& A\# A \ar[r]^-F& C\ar[r]^\sigma&C \ar@{=}[d]  \\
    t \ar@{=}[d]\ar[r]^-{\delta_t} & t\#t \ar[r]^-{a_0\#a_0} &A\#A \ar@{=}[d]\ar[rr]^F&&C \ar@{=}[d]\\
    t \ar[r]_{a_0} & A \ar[r]_-{\delta_A} & A\#A\ar[rr]_F & & C 
  }
\]
where we have used the defining property of $a_0$ in the top rectangle, and naturality of $\delta$ in the bottom left rectangle.
\end{proof}

We can also do a version of the fixed-point theorem, taking into account the Yanofsky variation, involving the map $p\colon B\to A$. But here we do not need to assume $p$ is split, but merely that it is \emph{surjective on $t$-points}: every $a\colon t\to A$ lifts through $p$ to a $b_a\colon t\to B$ satisfying $p\circ b_a=a$.

\begin{theorem}\label{fixed_point_thm_Yanofsky_variation}
Let $(\CC,\#,\delta,t)$ be a pointed magmoidal category with diagonals, $p\colon B\to A$ a map that is surjective on $t$-points, and $F\colon A\#B\to C$ and $\sigma\colon C\to C$ be maps such that 
\begin{align*}
&\exists a\colon t\to A\\
&\forall b\colon t\to B\\
&\sigma\circ F\circ (p\# \id)\circ\delta_B\circ b = F\circ a\#b\circ \delta_t.
\end{align*}
Then for any lift $b_a\colon t\to B$ of $a$, the arrow $c := F\circ (p\# \id)\circ\delta_B\circ b_a\colon t\to C$ satisfies $\sigma \circ c=\sigma$.
\end{theorem}

\begin{proof}
Given $a$ as in the hypotheses, take any lift $b_a$ of it.
Then the following diagram commutes
\[
    \xymatrix{
     B \ar[r]^-{\delta_B} & B \# B \ar[r]^{p\# \id} & A\# B \ar[r]^-{F} & C \ar[r]^\sigma &  C \ar@{=}[d]\\
    t\ar[u]^{b_a} \ar@{=}[d] \ar[r]^-{\delta_t} &  t\# t \ar[rr]^{a\# b_a} \ar@{=}[d] && A \# B \ar[r]^F \ar@{=}[d] &C \ar@{=}[d]\\
    t \ar@{=}[d] \ar[r]^-{\delta_t} & t\# t \ar[r]^-{b_a\# b_a} & B\# B \ar[r]^{p\# \id} \ar@{=}[d] & A\# B \ar[r]^-{F} &  C\ar@{=}[d] \\
    t \ar[r]_{b_a} & B\ar[r]_-{\delta_B} & B\# B \ar[r]_{p\# \id} & A\# B \ar[r]_-{F} &  C
    }
\]
proving the claim. Note that naturality of $t$ has only been used here with respect to $b_a$.
\end{proof}

\section{Reduction to a concrete category in the cartesian setting}\label{sec:concrete_reduction}

If one reads Yanofsky's survey \cite{Yanofsky_03} of results flowing from Lawvere's theorems, then the two results are stated outright in terms of sets and functions.
This is in stark contrast to Lawvere's austere, pure category-theoretic treatment, nowhere mentioning sets except in applications; Cantor's theorem on unboundedly-large cardinalities is a prototypical case.
Yanofsky is completely self-aware around this seeming lack of generality:

\begin{quote}
It is here that we get in trouble ignoring the category theory that is necessary. In the examples that we will do, the objects we will be dealing with have more structure than just sets and the functions between the objects are required to preserve that structure. \cite[Remark 4]{Yanofsky_03}
\end{quote}

Even so, there is still some distance between structured sets and structure-preserving functions, and a general cartesian category. 
However, if one examines what Lawvere's theorems say, it becomes apparent that they are actually theorems about the \emph{concrete quotient} of the given cartesian category.
This is because of the following construction.
Given a category $\CC$ and a fixed object $t$, the functor $\CC(t,-)\colon \CC\to \Set$ factors as a composite $\CC \xrightarrow{Q} \CC_{=_t} \xrightarrow{|-|} \Set$ where $Q$ is full and bijective on objects, and $|-|$ is faithful.
More concretely, we can define $\CC_\sim$ to have the same objects as $\CC$, and define a congruence $=_t$ on the morphisms of $\CC$, so that $f,g\colon X\to Y$ are equivalent precisely when, for all $x\colon t\to X$, $f\circ x = g\circ x$.
Thus we can see that $\CC$ is something like an \emph{extensional quotient}, except `elements' here means $t$-points: morphisms from the fixed object $t$.
This concept is more familiar in computer science, for instance, than classical mathematics, whereby different programs (considered as morphisms in some category, for instance) have the same output on all inputs. 
As a consequence of the construction of $\CC_{=_t}$, we have $\CC(t,X)\simeq \CC_{=_t}(Qt,QX)$ for all objects $X$ of $\CC$, and even stronger, $|-|=\CC_{=_t}(Qt,-)$ as functors.
It is a short exercise to check that if $\CC$ is a cartesian category, then $\CC_{=_t}$ is a cartesian category, and $Q$ and $|-|$ preserve finite products.
For a cartesian category, then there is a canonical concrete category, also cartesian, whose objects can be considered as sets with structure, and morphisms are functions preserving that structure. 
Even better, the underlying set of a product in this concrete category is the product of the underlying sets, a fact which doesn't hold in arbitrary cartesian concrete categories.

The statements of Lawvere's diagonal argument and fixed point theorem both only involve global points of cartesian categories (hence taking $t$ to be a terminal object $1$), and so are really statements about the concrete cartesian category $\CC_{=_1}$.
Thus the apparent lack of generality of Yanofsky's treatment, modulo the acknowledged reduced emphasis on sets-with-structure, is no real restriction at all.

The possibility of using generalised elements is not new. For example Mulry \cite{Mulry_89}, in working up to a Lawvere-style fixed-point theorem, also allows $t$-points instead of global elements, calling them \emph{$t$-paths}.\footnote{Thinking of the prototypical case $t=\NN$, and $\NN\to X$ as defining a `path' in $X$.} 
Specifically, Mulry is interested in the case that $t=\NN$ is a natural numbers object or similar, referring to a map $f\colon X\to Y$ being $t$-path surjective when $f_*\colon \CC(t,X)\to \CC(t,Y)$ is a surjective function. 
However, this concept is mostly just used in the background setup constructions, but not at the point when Lawvere's fixed-point theorem is invoked.
Here, working up to ``observational equivalence'' in Lawvere's setting reduces his arbitrary cartesian category to the case of sets-with-structure, where the forgetful functor to $\Set$ preserves products.

\section{The fixed-point theorem in the internal logic of a regular category}\label{sec:internal_logic_regular_cat}

The logical structure of the fixed point theorem is so limited that we can reason in the internal logic of a \emph{regular} category \cite{Butz_98}, which is the regular fragment of first-order logic, dealing with just conjunction $\wedge$, truth $\top$, and existential quantifiers $\exists$. 
Working in the internal logic of a category is a generalisation of working with Kripke semantics, namely a rich alternative semantics whereby existential quantifiers are instantiated by passing to a slice category via the pullback functor associated to a regular epimorphism.
As we have seen above, the fixed point theorem can be written in such a way so that there is one (bounded) existential quantifier, namely over the `set' of indices of maps. 
The data going into the theorem statement is a two-variable function and an endomorphism of its domain. 
The existential quantifier means passing up some regular epimorphism $t\onto 1$ and then getting a $t$-point.
The fixed point is then constructed algebraically from this, as before.

The following theorem is the translation\footnote{You can think of this as compiling a high-level description into a concrete statement in the language of categories.} of the fixed-point theorem into the internal logic of a regular category.

\begin{theorem}\label{thm:FPT_internal_reg_cat}
Let $\CC$ be a regular category, and fix a pair of morphisms $F\colon A\times A\to C$ and $\sigma\colon C\to C$. 
Assume there exists a morphism $a_0\colon t\to A$ such that  $!_t\colon t\onto 1$ is a regular epimorphism and the diagram\footnote{The top row of this rectangle can be considered as the map $F(a_0,-)$.} 
\[
  \xymatrix{
    t\times A \ar[r]^-{a_0\times\id_A} \ar[d]_{\pr_2} & A\times A \ar[r]^-F & C \\
    A \ar[r]_-{\delta_A} &A\times A \ar[r]_-F& C \ar[u]_\sigma
  }
\]
commutes.
Then the composite $c:=\sigma\circ F\circ \delta_A\circ a_0\colon t\to C$ satisfies $\sigma\circ c = c$.
\end{theorem}

\begin{proof}
Consider the following commutative diagram:
\[
  \xymatrix{
    t \ar[d]_{\delta_t} \ar[rr]^{a_0} && A  \ar@{=}[d] \ar[r]^-{\delta_A} & A\times A  \ar[r]^-{F} & C \ar[r]^{\sigma} & C  \ar@{=}[d] \\
    t \times t \ar@{=}[d] \ar[r]_-{\id_t\times a_0} & t\times A \ar[d]^{a_0\times \id_A} \ar[r]_-{\pr_2} & A \ar[r]_-{\delta_A} & A\times A \ar[r]_-{F} & C \ar[r]_{\sigma} & C \ar@{=}[d] \\
    t \times t 
    \ar[r]_{a_0\times a_0} & A\times A \ar[rrrr]_F & &&& C\,. 
  }
\]
But as $\delta$ is a natural transformation, $(a_0\times a_0)\circ \delta_t = \delta_A\circ a_0$, so that the composite along the left and bottom edges is again $c = \sigma\circ F\circ \delta_A\circ a_0$, as required.
\end{proof}

If $\CC$ is well-pointed, in the sense that $1$ is regular-projective and a separator, then $\CC$ is concrete, and the regular epimorphism $t\onto 1$ has a section $1\to t$.
Then the $t$-point $a$ in the proof gives rise to a global point $1\to t \to A$, and we recover Lawvere's fixed-point theorem.\footnote{One question is whether one can do the diagonal theorem in a similar weak internal logic setting. Here it is trickier, since one has to quantify over all points of a `set' with an endomorphism (that is, the endo is free). But then there is no existential quantifier!}
However, the theorem as formulated like this is an unravelling of the direct translation of Lawvere's theorem into the internal language of $\CC$.

Notice that the fact the top left rectangle commutes depends only on the naturality of the projection morphism $\pr_2$ (here, with respect to $a_0$), and the identity $\pr_2\circ \delta =\id$.
In particular it doesn't use the \emph{left projection} $\pr_1$ at all, and so this is superfluous data.
We can then consider magmoidal categories with diagonal and \emph{right projection} only, where right projection means a system of natural maps $\pr_2\colon A\#B\to B$ such that $\pr_2\circ \delta = \id$. 
Such magmoidal categories do not need to be symmetric, and indeed, symmetry is not used in the proof of Theorem~\ref{thm:FPT_internal_reg_cat}.
If one has a magmoidal category with diagonal and right \emph{and} left projections, then the magmoidal structure is automatically cartesian.

\begin{example}
Recall that given a monoidal category $(\CC,\otimes)$ with a pointed endofunctor $T$, there is a magmoidal struture on $\CC$ defined by $A{}_T\# B := T(A)\otimes B$.
If the monoidal structure is cartesian, then, in addition to diagonals, the magmoidal structure has natural projection maps $\pr_2\colon A{}_T\# B\to B$.
\end{example}

One should compare this with the setting of ordered logic without Exchange, where one could postulate that Weakening is allowed only on one side of a multiplicative conjunction.

\begin{remark}
Abramsky and Zvesper \cite{Abramsky_Zvesper} give the Brandenburger--Keisler argument in epistemic game theory in a form suitable for the internal logic of a regular category; it is derived from a different generalisation of Lawvere's fixed point theorem for regular categories than what is considered in this note.
\end{remark}

\section{Getting a uniform construction of the fixed point}\label{sec:uniform}

If we are in a (right) closed magmoidal category with diagonals, then we can make the construction of the fixed point to be an actual higher-order function.
That is, if the existential in the fixed-point theorem is replaced by a morphism, $C^C \to A$, which we can interpret as sending an endomorphism $\sigma$ to the `index' $a_0\in A$ so that $\sigma\circ F_\delta\circ a = F\circ a_0\#a\circ \delta_t$. More generally, this could also depend on $F$.
The end result should be a function $C^C\to C$ returning the fixed point of each endomorphism. 

In fact we can be a little more flexible; it might be that there is a function picking the index $a_0$, but it is ``discontinuous'' or ``non-uniform'' in its argument.
This can be achieved by having the ambient category $\CC$ be equipped with a copointed endofunctor $\flat\colon \CC\to \CC$, whose definition is recalled below.
Such an endofunctor implements a necessity modality satisfying the `T' axioms\footnote{Written $\square$ in the modal logic literature, with axioms dating back to Feys \cite{Feys_37}, and independently von Wright \cite[appendix II]{vonWright_51}, there called `M'.}, crucially that one has a natural map $e_X\colon \flat(X)\to X$. 
The existence of this map is the analogue of the \emph{axiom of necessity}.
One should also compare with the treatment by Restall \cite{Restal_93} of modalities in the substructural setting.

\begin{definition}
Let $(\CC,\#)$ be a magmoidal category. An endofunctor $\flat\colon \CC\to \CC$ is \emph{copointed} if it is equipped with a natural transformation $e_X\colon \flat(X) \to X$, called the \emph{counit} of $\flat$.
\end{definition}

\begin{example}
Any \emph{idempotent comonad}, in particular Shulman's flat modality \cite{Shulman_18}, has an underlying counital endofunctor. 
Recall that in addition to the counit, an idempotent comonad has a natural isomorphism $m_X\colon \flat(X)\to \flat(\flat(X))$,\footnote{The natural maps $e$ and $m$ can be seen to be the analogues of the S4 necessity axioms.} such that
\begin{itemize}
    \item[(CA)] $m_{\flat X}\circ m_X = \flat(m_X)\circ m_X$
    \item[(CU)] $e_{\flat X}\circ m_X = \flat(e_X) \circ m_X = \id_{\flat(X)}$
\end{itemize}
Shulman \cite{Shulman_18} implements `crisp' variables in real-cohesive type theory using such a comonad (see also \cite{LOPS_18}).\footnote{The `flat modality' has also been implemented in the formal proof assistant Agda \cite{agda_flat}.}
\end{example}

One possibility is that $\flat=\id_{\CC}$, and $e_X$ is always an identity map in $\CC$ (likewise $m_X$, if it is present), so that Theorem~\ref{thm:uniform_FPT} incorporates as a degenerate case the purely uniform version.
For ease of reference, I will adopt Shulman's terminology and call a morphism $\flat(X)\to Y$ a \emph{crisp} morphism from $X$ to $Y$, or just \emph{crisp}.

\begin{example}
Given $\disc\colon \Top\to \Top$ the endofunctor that takes a topological space and returns the discrete space on the same underlying set. 
Then a crisp function $\disc(A)\to B$ is an arbitrary, possibly discontinuous function. 
This is the example that gives rise to the intuition behind crisp morphisms.
Similar examples are possible, when restricting to subcategories like that of manifolds.
\end{example}

While this example seems far from the realm of the computable, recall the situation in type theory where one considers for example computable functionals $(\NN \to \Bool) \to \Bool$ that act as if they were continuous maps, for the (non-discrete) product topology on the domain.\footnote{Recall that the function type $\NN\to \Bool$ corresponds under topological semantics to the product of $\NN$-many copies of $\Bool$.}
The interpretation of $\flat$ as returning a `discretised' type would then allow `non-computable' functionals.
More generally, one might profitably think of a term of type $\flat(A)$ as being a term of type $A$ but supplied by an \emph{oracle}, and a function $\flat(A)\to B$ as taking oracular inputs and returning terms of type $B$. 
The counit $\flat(A)\to A$ then tells us that if an oracle supplies a term of type $A$ then we have that term, but it's not computable.
An analogous example in a different flavour of substructural logic is the exponential (or `of-course') modality $!$ in linear logic, which is in most formalisms encoded in semantics by a comonad, but is at the very least given by a copointed endofunctor. 
Linear logic famously does not allow the Contraction structural rule (modelled semantically by diagonals), so is inherently `resource sensitive'---except if one has a function $!A\to B$, which models allowing arbitrary duplication of inputs from $A$.
Our semantic model of a magmoidal category with diagonals clearly allows some form of `resource duplication', if a weak one, so here the idea of $\flat$ is meant to model more `resources' in the form of an `oracle'.

Given the data of a copointed endofunctor, considered as a fragment of a comonadic modality, the index function can be $\flat(C^C)\to A$, so that $a$ is a \emph{crisp} function of $\sigma$, or the section of $p\colon B\to A$ could instead be $s\colon \flat(A)\to B$, with $p\circ s = e_A$, so that $b_a$ as in Theorem~\ref{fixed_point_thm_Yanofsky_variation} is a function of $a$, but a \emph{crisp} one.
In this setting, we get a morphism $\flat(C^C)\to C$ that can be seen as sending an endomorphism to its fixed point, but in a crisp way.
This is reminiscent of the version of Brouwer's fixed point theorem in real-cohesive type theory \cite{Shulman_18}, whereby the fixed point becomes a crisp function of the endomorphism .

\begin{theorem}\label{thm:uniform_FPT}
Let $(\CC,\#,\delta)$ be a magmoidal category with diagonals, equipped with a copointed endofunctor $\flat$. Assume given an epimorphism $p\colon B\to A$ with section $s$, and morphisms $F\colon A\times B\to C$ such that the object $C^C$ exists, with evaluation map $\ev\colon C^C\#C\to C$. 

If there is a morphism $\idx^\flat \colon \flat(C^C)\to A$ such that
\[
    \xymatrix{
        \flat(C^C)\# B \ar[rr]^-{\idx^\flat\#\id} \ar[d]_{e\#(p\#\id)\delta_B} && A\# B \ar[d]^F \\
        C^C \# (A\# B) \ar[r]_-{\id\# F} & C^C \# C \ar[r]_-{\ev} & C\,.
    }
\]
Then the crisp morphism $\fix^\flat$, defined as the composite
\[
    \xymatrix{
        \flat(C^C) \ar[r]^-{\idx^\flat} & A \ar[r]^-{\delta}  & A\#A\ar[rr]^{\id\#s} && A\#B \ar[r]^-F & C
    },
\]
returns the fixed point of an endomorphism of $C$, in the sense that
\[
    \xymatrix{
        \flat(C^C) \ar[rr]^-{\fix^\flat} \ar[d]_{\delta} && C\\
        \flat(C^C)\#\flat(C^C) \ar[rr]_-{e\# \fix^\flat} &&C^C\# C\ar[u]_{\ev}\,.
    }
\]
\end{theorem}

\begin{proof}
Write the endomorphism object as $E:=C^C$ for simplicity.
Consider the following diagram, which uses the assumption on $\idx^\flat$ for the bottom right rectangle:
\[\scalebox{0.85}{
  \centerline{
    \xymatrix{
        \flat(E) \ar[r]^-{\delta} \ar@{=}[d] & \flat(E)\# \flat(E) \ar@{=}[d] \ar[rrrrr]^-{e\# \fix^\flat} &&&&& E\# C \ar[r]^-{\ev} & C \ar@{=}[d]\\
        \flat(E) \ar[r]^-{\delta} \ar@{=}[dd] & \flat(E)\# \flat(E) \ar[r]^-{\id\#\idx^\flat} & \flat(E)\#A \ar[r]^-{e\#\delta} \ar@{=}[d] & E\#(A\#A)\ar[rr]^{\id\#(\id\#s)} \ar[dr]^{\id\#(s\#s)} && E\#(A\#B) \ar[r]^-{\id\#F} \ar@{=}[d] & E\#C\ar[r]^-{\ev}& C\ar@{=}[d]\\
         && \flat(E) \#A \ar[r]^-{\id\#s} \ar[d]_{\idx^\flat\#\id} &\flat(E)\# B \ar[r]_{e\#\delta} \ar[d]_{\idx^\flat\# \id}&E\#(B\#B) \ar[r]_{\id\#(p\#\id)}& E\#(A\#B) \ar[r]_-{\id\#F}& E\#C \ar[r]_-{\ev}& C\ar@{=}[d]\\
        \flat(E)\ar[r]_{\idx^\flat} & A \ar[r]_-{\delta} & A\# A \ar[r]_{\id\#s} & A\#B\ar[rrrr]_F &&&& C\,.
    }
  }
}\]
The bottom edge is just $\fix^\flat$ again, as needed.
I note in passing that naturality of $\delta$ is required with respect to $\idx^\flat$ and $s$.
\end{proof}

In the degenerate case where $\flat = \id$, we get a more familiar fixed-point combinator $C^C\to C$.

\begin{remark}\label{rem:uniform_FPT_variations}
If one has the data of a \emph{crisp} section $s^\flat\colon \flat(A)\to B$, then much the same argument applies if $\flat$ is assumed to be a comonad. We can define $\fix^\flat$ in this case as
\[
    \xymatrix{
        \flat(C^C) \ar[r]^-m &\flat\flat(C^C)  \ar[r]^-{\flat(\idx^\flat)} & \flat(A) \ar[rr]^-{(e\#s^\flat)\circ\delta}  &  & A\#B \ar[r]^-F & C
    }.
\]
The only change is some judicious insertions of $m$, so as to get composites of the form $\flat(C^C)\xrightarrow{m}\flat\flat(C^C)\xrightarrow{\flat(\idx^\flat)} \flat(A)\xrightarrow{s^\flat} B$.

Conversely, if one has crisp $s^\flat$ but plain $\idx\colon C^C\to A$, then again a version of $\fix^\flat$ can be defined.
In this case, $m$ is not needed, since it suffices to consider composites of the form $\flat(C^C)\xrightarrow{\flat(\idx)} \flat(A)\xrightarrow{s^\flat}A$.
\end{remark}

\section{Coda}

On paper, the distance between the fiat existence of $\idx^\flat$ and the proved existence of $\fix^\flat$ does not seem like much. 
Or to put it another way, the assumption on $\idx^\flat$ seems close to what we are aiming to prove at the end, so what simpler assumption(s) can be used to construct $\idx^\flat$?
Here's an example, leaving out the complication of the split surjection $p\colon B\to A$ and the crisp maps, that is closer in spirit to existing approaches.

\begin{theorem}\label{thm:uniform_FPT_from_split_epi}
Let $(\CC,\#,\delta)$ be a magmoidal category with diagonals.
Assume given objects $A,C$ such that the internal hom objects $C^C$ and $C^A$ exist, and a map $\alpha\colon A\onto C^A$, with a section $\ell\colon C^A\to A$. 
Then there is a map $\fix\colon C^C\to C$ returning a fixed point.
\end{theorem}

\begin{proof}
Define the map $F$ as 
\[
    A\# A \xrightarrow{\alpha\# \id} C^A \# A \xrightarrow{\ev} C,
\]
and $\idx$ as
\[
    C^C \xrightarrow{(F\circ \delta)^*} C^A \xrightarrow{\ell} A\, .
\]
Here $(F\circ \delta)^*$ corresponds to the map
\[
    \xymatrix{
    C^C \# A  \ar[r]^-{\id\# \delta} &
    C^C\# (A\# A) \ar[r]^-{\id\# F} & C^C \# C \ar[r]^-{\ev} & C
    }
\]
and is the unique morphism making the square
\[
    \xymatrix{
    C^C \# A \ar[rrr]^{(F\circ \delta)^*\#\id} \ar[d]_{\id\# \delta} &&& C^A \# A \ar[d]^{\ev}\\
    C^C\# (A\# A) \ar[rr]_-{\id\# F} && C^C \# C \ar[r]_-{\ev} & C
    }
\]
commute, by the defining property of $C^A$.
Inserting $\id\#\id = \id\# \alpha\circ \ell\colon C^A \# A \to A\# A \to C^A\# A$ into this square, this shows that the hypothesis on $\idx$ from Theorem~\ref{thm:uniform_FPT} holds, taking $B=A$, and $p=s=\id$.
We can construct the resulting map $\fix$ explicitly as
\[
    \fix\colon C^C \xrightarrow{(F\circ \delta)^*} C^A \xrightarrow{\ell} A \xrightarrow{F\circ \delta} C.\qedhere
\]
\end{proof}

\begin{remark}
Escard\'o \cite{Escardo_LFPT} has written a formalized treatment of Lawvere's fixed-point theorem using the Agda proof assistant. It is similar in form to the previous theorem, albeit in the language of Martin-L\"of type theory (MLTT), and from a category theoretic point point of view can be seen as taking place in the syntactic category of MLTT.
In fact two versions are given, one with an explicit splitting of what is here written as $\alpha$, and one where $\alpha$ is only taken as surjective (where this has a specific definition in type theory).
\end{remark}

\begin{remark}
One can have a version of the previous theorem in the presence of a copointed endofunctor $\flat$, using instead of a section $\ell$, a crisp map $\ell^\flat\colon \flat(C^A)\to A$ satisfying $\alpha\circ \ell^\flat = e_{C^A}$.
This gives crisp $\fix^\flat\colon \flat(C^C)\to C$ as before, the proof of which is left as an exercise for the reader.
\end{remark}

Taking $A=C$ in the theorem, we get an even more familiar-seeming result, albeit in the setting of a magmoidal category with diagonals. 
Recall that we say $C$ is a \emph{reflexive object} if $C^C$ exists and there is an epimorphism $\app\colon C\to C^C$ with a section $\lam\colon C^C\to C$ . The classic example, in the special case of a cartesian closed category, arises from $\lambda$-calculus (see eg \cite{Scott_80}).

\begin{corollary}\label{cor:uniform_FPT_from_reflexive}
Given a magmoidal category $(\CC,\#,\delta)$ with diagonals and given a \emph{reflexive object} $(C,\app,\lam)$, we have the fixed-point map\footnote{The map $\mathrm{uncurry}(\app)$ corresponds to $F$ above, since $\app$ is the curried form of $F$.} 
\[
    \fix\colon C^C\xrightarrow{(\mathrm{uncurry}(\app)\circ\delta)^*} C^C \xrightarrow{\lam} C \xrightarrow{\mathrm{uncurry}(\app)\circ\delta} C.
\]
\end{corollary}

Reflexive objects in \emph{cartesian closed} categories are intimately related to models of $\lambda$-calculus \cite{Scott_80}, but it is not yet clear what corresponds to these more general reflexive objects in magmoidal categories---ideally some substructural analogue of $\lambda$-calculus, linked to the $\mathit{BW}$-logic coming from combinatory algebra.

{\raggedright
\printbibliography}

\onecolumn\newpage

\end{document}